\documentclass[11pt, reqno, a4paper]{amsart}
\usepackage{amsmath}
\usepackage{amsfonts}
\usepackage{amssymb}
\usepackage{amsthm}
\usepackage{newlfont}
\usepackage{verbatim}
\theoremstyle{plain}
\newtheorem{thm}{Theorem}

\newtheorem{prop}[thm]{Proposition}

\theoremstyle{plain}

\theoremstyle{remark}
\newtheorem{rem}[thm]{Remark}
\newcommand{\R}{\mathbb{R}}
\newcommand{\ep}{\varepsilon}
\newcommand{\rd}{\partial}

\title{Splitting of 3-Manifolds and \\
Rigidity of Area-Minimising Surfaces}

\author{Mario Micallef}
\address{Mathematics Institute, University of Warwick,
Coventry, CV4 7AL, U.K.}
\email{M.J.Micallef@warwick.ac.uk}
\author{Vlad Moraru}
\address{Mathematics Institute, University of Warwick,
Coventry, CV4 7AL, U.K.}
\email{vmoraru@gmail.com}

\begin{document}
\begin{abstract}
In this paper we modify an argument in \cite{BBN}
to prove an area comparison result (Theorem \ref{areadec})
for certain totally geodesic surfaces in 3-manifolds
with a lower bound on the scalar curvature. This theorem
is a variant of a comparison theorem (Theorem 3.2 (d) in
\cite{HK}) of Heintze-Karcher for minimal hypersurfaces in manifolds
of nonnegative Ricci curvature. Our assumptions on the ambient manifold
are weaker but the assumptions on the surface are
considerably more restrictive. We then use our comparison theorem
to provide a unified proof of various splitting theorems
for 3-manifolds with lower bounds on the scalar curvature
that were first proved in \cite{CG}, \cite{BBN} and \cite{N}.
\end{abstract}

\maketitle
\section{Introduction and statement of main results}
In \cite{Si}, Corollary 3.6.1, Simons observed that
there are no closed, stable, minimal, 2-sided hypersurfaces
in a manifold of positive Ricci curvature.
This is a variant of the classical Synge lemma.
An easy, but unstated, extension of Simons's observation is that
a closed, stable, minimal, 2-sided hypersurface $\Sigma$ in a manifold $M$
of nonnegative Ricci curvature is necessarily totally geodesic
and the normal Ricci curvature of $M$ must vanish all along $\Sigma$.
The simplest such example is $\Sigma \times (-\ep,\ep)$, $\ep > 0$,
with the product metric $g + dt^2$,
where the metric $g$ on $\Sigma$ has nonnegative Ricci curvature.
However, the existence of a closed, stable, minimal, 2-sided
hypersurface $\Sigma$ in a manifold $M$ of nonnegative Ricci curvature
does not imply that the metric of $M$ near $\Sigma$ must split as
$g + dt^2$. Indeed, the metric $(1-t^4)g + dt^2$ on
$\Sigma \times (-\ep,\ep)$, $1> \ep > 0$,
has nonnegative Ricci curvature if $g$ has nonnegative Ricci curvature
and $\Sigma \times \{0\}$ is stable.\footnote{If $g$ is the round metric
on the unit sphere $S^n$, then it is easy to embed this example,
for $\ep$ sufficiently small,
into a metric on $S^{n+1}$ of nonnegative sectional curvature.}
However note that in this example,
$\Sigma \times \{0\}$ does not minimise area, not even locally.

So, one might surmise whether the existence of a closed,
\emph{area-minimising}, 2-sided hypersurface $\Sigma$
in a manifold $M$ of nonnegative Ricci curvature
implies that the metric of $M$ near $\Sigma$ must split as
$g + dt^2$. This is indeed the case and it follows from
a special case of Theorem 3.2 (d) in \cite{HK},
where Heintze and Karcher prove that
the exponential map of the normal bundle
$\Sigma \times \R$ of $\Sigma$ in $M$ is volume non-increasing when
$M$ has nonnegative Ricci curvature,
$\Sigma$ is two-sided and minimal in $M$ and
$\Sigma \times \R$  is equipped with the product metric.
Anderson extended this result to area-minimising integral currents
of codimension 1 in a compact manifold of nonnegative Ricci curvature;
see Theorem 3 in \cite{An}. Anderson then used this theorem,
together with an existence result for area-minimising hypersurfaces,
to obtain a different proof of the splitting theorem in \cite{CHG}
of Cheeger and Gromoll in the compact case;
see Corollary 3 in \cite{An}.

It is easy to construct examples of totally geodesic,
area-minimising hypersurfaces in manifolds of positive scalar curvature.
Therefore, positive Ricci curvature in Simons's result
cannot be relaxed to positive scalar curvature.
Nevertheless, in the well-known paper \cite{SY}, Schoen and Yau
obtained topological restrictions on an oriented, two-sided,
stable, minimal surface $\Sigma$ in a 3-manifold $M$
whose scalar curvature $R$ is positive. In particular,
they proved that the genus of $\Sigma$ must be zero. Soon after,
Fischer-Colbrie and Schoen studied the case $R \geqslant 0$
and proved in \cite{FCS} that, in this case, the genus of $\Sigma$
must be zero or one, and if it is one,
then $\Sigma$ is totally geodesic and flat and both
the normal Ricci curvature of $M$ and $R$ vanish all along $\Sigma$.

A closer look at the proof of Schoen-Yau reveals that
a lower bound on the scalar curvature $R$ of the 3-manifold $M$
provides a bound on the area of the stable minimal surface $\Sigma$.
More precisely,
\begin{enumerate}
\item if $R \geqslant 2$
then the area of $\Sigma$ is bounded above by $4 \pi$ and
\item if $R \geqslant -2$ then the area of $\Sigma$
is bounded below by $4 \pi(\gamma-1)$
where $\gamma$ is the genus of $\Sigma$.
\end{enumerate}
To our knowledge, these bounds first appeared in \cite{SZ}.
Easy examples show that no area bounds are possible
for stable minimal tori in flat three-dimensional tori.
If the area bound in (1) or (2) is attained and
$R_0 := \min_{x \in M}R(x)$, then
an analysis similar to that used by Fischer-Colbrie and Schoen
in the case of genus one yields that
\begin{enumerate}
\item[(i)] the surface is totally geodesic,
\item[(ii)] $R$ is equal to $R_0$ all along $\Sigma$ and
\item[(iii)] the normal Ricci curvature of $M$
vanishes all along $\Sigma$.
\end{enumerate}
It follows that
\begin{enumerate}
\item[(iv)] the Gauss curvature $K$ of $\Sigma$
is identically equal to $\tfrac12R_0$.
\end{enumerate}

The obvious examples in which the area bound (1) or (2) is attained
are provided by $\Sigma \times (-\ep,\ep)$, $\ep > 0$,
with the product metric $g + dt^2$,
where $g$ has constant Gauss curvature equal to $\frac12 R_0$.
A natural question is whether these are the only examples in which
equality is attained in (1) and (2). As in the case of
nonnegative Ricci curvature mentioned earlier, this does not hold.
Indeed, the scalar curvature $R$ of the metric
$(1-t^4)g + dt^2$ satisfies $R \geqslant R_0$ and
$\Sigma \times \{0\}$ is stable but, once again,
does not minimise area, not even locally.

So again, one surmises whether the metric near
an \emph{area-minimising} closed surface $\Sigma$
which satisfies (i), (ii) and (iii), in a 3-manifold $M$
whose scalar curvature $R$ is greater than or equal to $R_0$, splits.
The examples of totally geodesic, area-minimising hypersurfaces
in manifolds of positive scalar curvature that come easily to mind
do not satisfy condition (ii).

To make use of the area-minimising property,
one has to perturb $\Sigma$ in a way which decreases its area.
When the Ricci curvature of $M$ is nonnegative,
Heintze and Karcher (and, to second order, Simons)
showed that this is achieved by surfaces that are
equidistant from $\Sigma$. However, in our case
we do not have information on the Ricci tensor away from $\Sigma$
and so, this is not a suitable perturbation. We shall see that
the right thing to do is to move $\Sigma$ so that
it still has constant mean curvature. This turns out to be possible by
Proposition 3.2 in \cite{BBN} which essentially asserts the existence of
a one-parameter family of constant mean curvature surfaces
in a neighborhood of a surface $\Sigma$
which satisfies (i) and (iii); see also \cite{ACG}.
A detailed proof of the following statement,
based on the implicit function theorem, can be found in \cite{N}.

% CMC foliation --------------------------------------------------

\begin{prop}\label{cmc}
  Let $\nu$ be a unit normal field on an oriented, two-sided surface
  $\Sigma$ immersed in a three-manifold $M$.
  If $\Sigma$ satisfies (i) and (iii),
  then there exists $\ep > 0$ and a smooth function
  $w \colon \Sigma\times(-\ep,\ep) \rightarrow \R$
  such that, for all $t \in(-\ep,\ep)$ the surfaces
  \begin{equation}\nonumber
    \Sigma_t:=\{\exp_x(w(x,t)\nu(x)) : x \in \Sigma \}
  \end{equation}
  have constant mean curvature $H(t)$. Moreover we have
  \begin{equation}\nonumber
    w(x,0)=0, \quad
    \left.\frac{\rd}{\rd t}w(x,t)\right|_{t=0}=1,
    \quad \text{and} \quad
    \int_{\Sigma} (w(\cdot,t) - t) \, dA = 0,
  \end{equation}
  for all $x \in \Sigma$ and $t \in(-\ep,\ep)$.
\end{prop}

% Area decrease --------------------------------------------------

We shall prove the following theorem which,
in light of the preceding discussion,
is the optimal analogue of the
Heintze-Karcher Theorem 3.2(d) in \cite{HK} in the context of
3-manifolds with lower bounds on scalar curvature.
The proof relies heavily on the Gauss-Bonnet Theorem and so,
it is not at all clear how
it may be generalised to dimension 4 or higher.
\begin{thm}\label{areadec}
Let $M$ be a three-manifold and let $\Sigma \subset M$ be an immersed,
2-sided, closed, surface. Denote by $R$ the scalar curvature of $M$,
and let $R_0 := \min_{x \in M}R(x)$. Suppose that $\Sigma$
has the following properties.
\begin{enumerate}
\item[(i)] $\Sigma$ is totally geodesic,
\item[(ii)] $R$ is equal to $R_0$ all along $\Sigma$ and
\item[(iii)] the normal Ricci curvature of $M$
 vanishes all along $\Sigma$.
\end{enumerate}
Let $\Sigma_t$ and $\ep$ be as in Proposition \ref{cmc} and
denote by $A(t)$ the area of $\Sigma_t$. Then
there exists $0 < \delta < \ep$ such that
\[
\text{for }|t| < \delta, \quad
A(t) \leqslant A(0) = \text{ area of $\Sigma$.}
\]
$\Sigma$ has constant Gauss curvature equal to $\frac12 R_0$
and therefore, by Gauss-Bonnet, $A(0) = \frac{8 \pi}{R_0} |\gamma - 1|$
if $R_0$ is nonzero.
\end{thm}

% Splitting theorem ----------------------------------------------
Theorem \ref{areadec} is the key ingredient that we use to provide 
a unified and more elementary proof of the following splitting and rigidity theorem, 
the three cases of which were separately proved, 
using different techniques, in \cite{BBN}, \cite{CG} and \cite{N}. 

\begin{thm} \label{mainthm}
Let $(M,g)$ be a complete Riemannian three-manifold
with scalar curvature $R$ and let $R_0 := \min_{x \in M}R(x)$.
Assume that $M$ contains a closed, embedded,
oriented, two-sided, area minimizing surface $\Sigma$.
\begin{enumerate}
\item Suppose that $R_0 = 2$ and that the area of
$\Sigma$ is equal to $4 \pi$. Then $\Sigma$ has genus zero and
it has a neighbourhood which is isometric to the product
$g_1 + dt^2$ on $S^2 \times (-\delta,\delta)$ where $g_1$ is the metric
on the Euclidean two-sphere of radius 1.
\item Suppose that $R_0 = 0$ and that $\Sigma$ has genus one.
Then $\Sigma$ has a neighbourhood which is flat and
isometric to the product $g_0 + dt^2$ on $T^2 \times (-\delta,\delta)$
where $g_0$ is a flat metric on the torus $T^2$.
\item Suppose that $R_0 = -2$ and
that $\Sigma$ has genus $\gamma \geqslant 2$ and
area equal to $4 \pi (\gamma - 1)$.
Then $\Sigma$ has a neighbourhood which is isometric to the product
$g_{-1} + dt^2$ on $\Sigma \times (-\delta,\delta)$ where $g_{-1}$ is
a metric of constant Gauss curvature equal to $-1$ on $\Sigma$.
\end{enumerate}
\end{thm}

% THE PROOFS -----------------------------------------------------

\section{The proofs}
\begin{proof}[Proof of Theorem \ref{areadec}]
With the notation as in Proposition \ref{cmc}, let
\[
f_t(x) := \exp_x(w(x,t)\nu(x)),
\quad x \in \Sigma, \ t \in (-\ep,\ep).
\]
Thus, $f_0 =: f$ is the given totally geodesic embedding.
The \emph{lapse function}
$\rho_t \colon \Sigma \rightarrow \R$ is defined by
\begin{equation}
  \rho_t(x) := \langle \nu_t(x),\frac{\rd}{\rd t}f_t(x) \rangle,
\end{equation}
where $\nu_t$ is a unit normal to $\Sigma_t$,
chosen so as to be continuous in $t$.
It satisfies the following Jacobi equation (cf. \cite{HI} eq. (1.2))
\begin{equation}
  H'(t) = -\Delta_{\Sigma_t}\rho_t -
  (Ric(\nu_t,\nu_t) + \|B_t\|^2)\rho_t,
  \label{Jacobi-eq}
\end{equation}
where $B_t$ is the second fundamental form of $\Sigma_t$ and
$(\cdot)':=\frac{\rd}{\rd t}(\cdot)$.
Since $\rho_0(x) = 1$ for all $x \in \Sigma$, we can assume,
by decreasing $\ep$ if necessary,
that $\rho_t(x) > 0$ for all $x \in \Sigma$.
So we can divide \eqref{Jacobi-eq} by $\rho_t$ and,
on using the Gauss equation
\begin{equation}
    Ric(\nu,\nu) = \tfrac12 R - K + \tfrac12 H^2 - \tfrac12 \|B\|^2,
    \label{Gauss}
\end{equation}
\eqref{Jacobi-eq} becomes
\begin{equation}
  H'(t) \frac{1}{\rho_t} = -\frac{1}{\rho_t}\Delta_{\Sigma_t}\rho_t
  - \tfrac12 R_t + K_t - \tfrac12 H(t)^2 - \tfrac12 \|B_t\|^2
  \label{Jacobi-eq2}
\end{equation}
where $R_t(x) := R(f_t(x))$ and $K_t(x) := K(f_t(x))$
is the Gauss curvature of $\Sigma_t$ at $f_t(x)$.
We now make essential use of the hypotheses (i), (ii) and (iii) which,
via the Gauss equation \eqref{Gauss}, imply that
\begin{equation} \label{K0const}
K_0 \equiv \tfrac12 R_0 \leqslant \tfrac12 R_t(x) \quad
\forall \, x \in \Sigma, \ t \in (-\ep,\ep),
\end{equation}
and therefore, \eqref{Jacobi-eq2} can be rewritten as
\begin{align}\nonumber
 H'(t) \frac{1}{\rho_t} &= -\frac{1}{\rho_t}\Delta_{\Sigma_t}\rho_t
 + \tfrac12(R_0 - R_t) + (K_t - K_0)
 - \tfrac12 H(t)^2 - \tfrac12 \|B_t\|^2 \\
& \leqslant -\frac{1}{\rho_t}\Delta_{\Sigma_t}\rho_t + (K_t - K_0).
 \label{Jacobi-eq3}
\end{align}
We integrate \eqref{Jacobi-eq3} over $\Sigma_t$,
(by parts in the first term on the right) and, keeping in mind that $H$
does not depend on $x$ (hence the importance of using a
constant mean curvature perturbation), we obtain,
\begin{align}\nonumber
 H'(t) \int_{\Sigma}\frac{1}{\rho_t} \, dA_t
 & \leqslant -\int_{\Sigma}\frac{\|\nabla_t\rho_t\|^2}{\rho_t^2} \, dA_t
 + \int_{\Sigma} (K_t - K_0) \, dA_t \\[2\jot]
 & \leqslant  4 \pi(1-\gamma) - K_0 A(t) \,,
 \label{key-inequality}
\end{align}
where $dA_t$ denotes the area element of $\Sigma_t$
with respect to $f_t^*g$ and we have used
the Gauss-Bonnet theorem in the last inequality.

%---------------------------- CLAIM 1 ------------------------------

\bigskip
\noindent\textbf{Claim 1: }There exists
a positive real number $\delta < \ep$
such that $H(t) \leqslant 0$ for all $t\in[0,\delta)$.

\medskip
\begin{proof}[Proof of Claim 1]
There are three cases to consider.

\smallskip
\noindent \emph{Case 1. }$R_0 > 0$.
By scaling, we can arrange $R_0 = 2$.
Then by \eqref{K0const}, we have that $K_0 \equiv 1$,
and $\Sigma$ has genus zero and $A(0) = 4 \pi$.
Therefore inequality \eqref{key-inequality} becomes
\begin{align}\nonumber
 H'(t) \int_{\Sigma} \frac{1}{\rho_t} \, dA_t
 & \leqslant 4 \pi - A(t) = A(0) - A(t) = - \int_0^t A'(s) \, ds \\
 &= - \int_0^t \{ H(s) \int_{\Sigma} \rho_s \, dA_s \} \, ds ,
 \label{key-inequality2}
\end{align}
where, in the last equality, we have used
the first variation of area formula,
\begin{equation}
 A'(t) = \int_{\Sigma} H(t) \rho_t \, dA_t =
 H(t) \int_{\Sigma} \rho_t \, dA_t.
 \label{first-variation}
\end{equation}
Let $\phi(t) := \int_{\Sigma} \frac{1}{\rho_t} \, dA_t$ and
$\xi(t) := \int_{\Sigma} \rho_t \, dA_t$.
Since $\phi$ is strictly positive for all $t \in (-\ep,\ep)$,
inequality \eqref{key-inequality2} becomes
\begin{equation}
  \label{integro-differential}
  H'(t) \leqslant - \frac{1}{\phi(t)} \int_0^t H(s) \xi(s) \, ds.
\end{equation}

As mentioned above, $\rho_0 \equiv 1$ and, by continuity,
we may assume that $\frac12 < \rho_t(x) < 2 \
\forall \, t \in (-\ep,\ep)$ and $x \in \Sigma$.
Integrating over $\Sigma_t$ yields $\frac12 A(t) < \xi(t) < 2A(t)$.
On the other hand, by choosing $\ep > 0$ small enough,
we may assume that $\frac12 A(0) < A(t) < 2A(0)$ and hence,
$\frac14 A(0) < \xi(t) < 4A(0) \ \forall \, t \in (-\ep,\ep)$.
A similar argument holds for $\phi(t)$. In particular we have
\begin{equation} \label{bounds}
 \frac{1}{\phi(t)} < \frac{4}{A(0)} \quad\text{and}\quad
 \xi(t) < 4A(0), \ \forall \, t \in (-\ep,\ep).
\end{equation}

Suppose, for a contradiction, that there exists $t_+ \in (0,\delta)$
such that $H(t_+) > 0$. By continuity, $\exists \ t_- \in [0,t_+)$
such that $H(t_-) \leqslant H(t) \ \forall \, t \in [0,t_+]$.
Note that by \eqref{integro-differential} we must have
$H(t_-) < 0$. By the mean value theorem, $\exists \ t_1 \in (t_-,t_+)$
such that
\[
H'(t_1) = \frac{H(t_+) - H(t_-)}{t_+ - t_-}.
\]
So, by \eqref{integro-differential} and \eqref{bounds}, we have:
\[
\frac{H(t_+) - H(t_-)}{t_+ - t_-} = H'(t_1)
\leqslant - \frac{4}{A(0)} H(t_-) \big(4 A(0)\big) t_1 \,.
\]
It follows that
\[
H(t_+) \leqslant H(t_-) (1 - 16 \delta^2)
\]
which is a contradiction if $0 < \delta < \frac14$ because
$H(t_+) > 0$ and $H(t_-) < 0$.

\smallskip
\noindent \emph{Case 2. }$R_0 = 0$.
By \eqref{K0const}, we have that
$K_0 \equiv 0$ and $\Sigma$ has genus one.
So, inequality \eqref{key-inequality} becomes
$H'(t) \leqslant 0 \ \forall \, t \in [0, \ep)$ and therefore,
since $H(0) = 0$, $H(t) \leqslant 0 \ \forall \, t \in [0, \ep)$.

\smallskip
\noindent
\emph{Case 3. }$R_0 < 0$.
By scaling, we can arrange $R_0 = -2$.
Then by \eqref{K0const}, we have that $K_0 \equiv -1$,
and $\Sigma$ has genus $\gamma > 1$ and $A(0) = 4 \pi (\gamma - 1)$.
Therefore inequality \eqref{key-inequality} becomes
\begin{align}\nonumber
 H'(t) \int_{\Sigma} \frac{1}{\rho_t} \, dA_t
 & \leqslant A(t) - A(0) = \int_0^t A'(s) \, ds \\
 &= \int_0^t \{ H(s) \int_{\Sigma} \rho_s \, dA_s \} \, ds .
 \label{key-inequality3}
\end{align}

Assume, for a contradiction, that there exists $t_0 \in (0,\delta)$
such that $H(t_0) > 0$. Let
\[
 I := \{t \geqslant 0 : t \in [0,t_0], H(t) \geqslant H(t_0) \}.
\]

%---------------------------- CLAIM 2 ------------------------------

\noindent \textbf{Claim 2: } $\inf I=0$.
\begin{proof}[Proof of Claim 2]
Let $t^* := \inf I$ and assume, for a contradiction, that $t^*>0$.
By the mean value theorem, $\exists \ t_1 \in (0,t^*)$ such that
\begin{equation} \label{MVT}
  H(t^*) = H'(t_1)t^*,
\end{equation}
since $H(0)=0$. From \eqref{key-inequality3}, \eqref{bounds}
and \eqref{MVT} we have
\begin{align}
 H(t^*)&\leqslant \frac{t^*}{\phi(t_1)}
 \int_0^{t_1} H(s) \xi(s) \, ds\\ \nonumber
 &\leqslant \frac{t^*}{\phi(t_1)} \int_0^{t_1} H(t^*) \xi(s) \, ds
  \leqslant \frac{4t^*}{A(0)} H(t^*) (4 A(0) t_1) \\
 &< 16 H(t^*) \delta^2
\end{align}
which is a contradiction if $\delta < \frac14$ and
Claim 2 has been proved.
\end{proof}

%--------------------------- END OF CLAIM 2 -----------------------

Since $\inf I = 0$, it follows from the definition of $I$ that
$H(0) \geqslant H(t_0)$ and since, by assumption, $H(t_0)>0$,
we conclude that $H(0) > 0$. This contradicts the hypothesis that
$\Sigma$ is totally geodesic and the proof of Claim 1 is complete.
\end{proof}

%--------------------------- END OF CLAIM 1 -----------------------

We can now easily complete the proof of Theorem \ref{areadec}.
We have that $H(t) \leqslant 0 \ \forall \, t \in [0,\delta)$ and
therefore, \eqref{first-variation} implies that
$A'(t) \leqslant  0$. Hence
$A(t) \leqslant A(0) \ \forall \, t \in [0,\delta)$.
We can argue similarly for $t \in (-\delta, 0]$
to complete the proof of Theorem \ref{areadec}.
\end{proof}

The proof of Theorem \ref{mainthm} now follows easily by
a slight variation of arguments that appear in \cite{BBN} and \cite{N}.

\begin{proof}[Proof of Theorem \ref{mainthm}]
The conclusion of Theorem 2 and
the assumption that $\Sigma$ is area-minimising imply that,
for the CMC family of surfaces $\Sigma_t$
provided by Proposition \ref{cmc},
$A(t) = A(0) \ \forall \, t \in (-\delta,\delta)$.
In particular, each $\Sigma_t$ is area-minimising
and, if $\gamma \neq 1$, the area of each $\Sigma_t$ is
equal to $4 \pi |\gamma - 1|$.
It follows, from (i) in the Introduction that
each $\Sigma_t$ is totally geodesic.
This holds when $\gamma = 1$ as well.
Equation \eqref{Jacobi-eq} then tells us that
the lapse function $\rho_t$ is harmonic,
and therefore is constant on $\Sigma_t$,
i.e. $\rho_t$ is a function of $t$ only.

%---------------------------- CLAIM 3 ------------------------------

\noindent \textbf{Claim 3: } The vector field $\nu_t$ is parallel.
\begin{proof}[Proof of Claim 3] $\Sigma_t$ is totally geodesic and
therefore, $\nabla_{\frac{\rd f_t}{\rd x^i}}\nu_t = 0$.
\begin{align}\nonumber
  0 = \frac{\rd}{\rd x^i}\rho_t &= \langle
    \nabla_{\frac{\rd f_t}{\rd x^i}}\nu_t,\frac{\rd f_t}{\rd t} \rangle
    + \langle
    \nu_t,\nabla_{\frac{\rd f_t}{\rd x^i}} \frac{\rd f_t}{\rd t}
    \rangle\\ \nonumber
  & = \langle
    \nu_t,\nabla_{\frac{\rd f_t}{\rd x^i}},\frac{\rd f_t}{\rd t} \rangle
    \qquad (\Sigma_t \text{ is totally geodesic})\\ \nonumber
  & = \frac{\rd}{\rd t}\langle\nu_t,\frac{\rd f_t}{\rd x^i} \rangle
    - \langle\nabla_{\frac{\rd f_t}{\rd t}}\nu_t,\frac{\rd f_t}{\rd x^i}
    \rangle\\ \nonumber
  & = -\langle
    \nabla_{\frac{\rd f_t}{\rd t}}\nu_t, \frac{\rd f_t}{\rd x^i} \rangle.
\end{align}
Hence $\nabla_{\frac{\rd f_t}{\rd t}}\nu_t = 0$.
This, together with the fact that $\Sigma_t$ is totally geodesic,
implies that the vector field $\nu_t$ is parallel.
\end{proof}

%--------------------------- END OF CLAIM 3 -----------------------

It follows that the integral curves of $\nu_t$ are geodesics and
that the flow $\Phi$ of $\nu_t$ is just the exponential map,
i.e. $\Phi(t,x) = \exp_x(t \nu(x)) \ \forall \, x \in \Sigma$.
Furthermore, since $\nu_t$ is, in particular, a Killing field,
this exponential map $\exp_{( \cdot )}(t \nu(\cdot))$ is an isometry
for all $t \in (-\delta,\delta)$. In other words, if $g_{\Sigma}$ is
the restriction of $g$ to $\Sigma$ then the exponential map of
the $\delta$-neighbourhood $\Sigma \times (-\delta,\delta)$
of the zero section of the normal bundle of $\Sigma$ in $M$
with the metric $g_{\Sigma} + dt^2$ is an isometry onto its image.
\end{proof}

\begin{rem}
It is straightforward to show that $w(x,t) \equiv t$ and that
$\rho_t \equiv 1$. Let
$S_t := \{ \exp_x(t \nu(x)) : x \in \Sigma \}$.
Pick $t_0 \in (-\delta,\delta)$ and $x_0 \in \Sigma$
and set $w_0 := w(x_0,t_0)$. Then $\Sigma_{t_0}$ and
$S_{w_0}$ are both totally geodesic and touch at
$\exp_{x_0}(t_0 \nu(x_0))$. Therefore, they coincide.
In particular $w$ is a function of $t$ only and, since by
Proposition \ref{cmc}, $\int_{\Sigma} (w(\cdot,t) - t) \, dA = 0$,
we have $w(x,t) \equiv t$. It follows that
$\nu_t(x) = \frac{\rd}{\rd t}f_t(x)$ and that $\rho_t \equiv 1$.
\end{rem}

\section{Acknowledgements}
This research was partially supported by a
Warwick Postgraduate Research Scholarship (WPRS).

%----------------------------- bibl ------------------------------

\bibliographystyle{plain}

\end{document}